\documentclass[12pt,amsymb,fullpage]{amsart}
\usepackage{amssymb,amscd,pstricks}

\newtheorem{theorem}{Theorem}[section]
\newtheorem{defn}[theorem]{Definition}

\newtheorem{lemma}[theorem]{Lemma}

\newtheorem{eple}[theorem]{Example}
\newtheorem{rmk}[theorem]{Remarks}
\newtheorem{dsc}[theorem]{Discussion}
\newtheorem{nota}[theorem]{Notation}

\newsavebox{\indbin}
\savebox{\indbin}{\begin{picture}(0,0)
\newlength{\gnu}
\settowidth{\gnu}{$\smile$} \setlength{\unitlength}{.5\gnu}
\put(-1,-.65){$\smile$} \put(-.25,.1){$|$}
\end{picture}}

\newcommand{\be}{\begin{enumerate}}
\newcommand{\bd}{\begin{defn}}
\newcommand{\bt}{\begin{theorem}}
\newcommand{\bl}{\begin{lemma}}
\newcommand{\ee}{\end{enumerate}}
\newcommand{\ed}{\end{defn}}
\newcommand{\et}{\end{theorem}}
\newcommand{\el}{\end{lemma}}

\begin{document}
\title{Nonstandard Martingales, Markov Chains and the Heat Equation}
\author{Tristram de Piro}
\address{Mathematics Department, The University of Exeter, Exeter}
 \email{t.depiro@curvalinea.net}
\thanks{}
\begin{abstract}
We construct a nonstandard martingale from a discrete Markov chain. This is shown to be useful for solving the heat equation with a non smooth initial condition. We show that the nonstandard solution to the heat equation with a smooth initial condition specialises to the classical solution
\end{abstract}
\maketitle

One of the most fundamental results in the theory of Markov chains is the following;\\

\begin{theorem}
\label{coupling}
Let $P$ be the transition matrix of an irreducible, aperiodic,positive recurrent Markov chain, $\{X_{n}\}_{n\geq 0}$, with invariant distribution $\pi$. Then, for any initial distribution, $P(X_{n}=j)\rightarrow\pi_{j}$, as $n\rightarrow\infty$. In particular;\\

$p_{ij}^{(n)}\rightarrow \pi_{j}$, for all states $i,j$, as $n\rightarrow\infty$\\

\end{theorem}

\begin{proof}
A good reference for this result is \cite{N}. However, we give the proof as it is used and modified later. Let the initial distribution be $\lambda$, and let $I$ be the state space. Choose $\{Y_{n}\}_{n\geq 0}$, such that $\{X_{n}\}_{n\geq 0}$ and $\{Y_{n}\}_{n\geq 0}$ are independent, with $\{Y_{n}\}_{n\geq 0}$ Markov $(\pi,P)$. Let $T=inf\{n\geq 1:X_{n}=Y_{n}\}$. We claim that $P(T<\infty)=1$, $(*)$. Let $W_{n}=(X_{n},Y_{n})$. Then $\{W_{n}\}_{n\geq 0}$ is a Markov chain on $I\times I$. By independence, it has transition probabilities given by;\\

$\overline{p}_{(i,j)(k,l)}=p_{ik}p_{jl}$\\

and initial distribution $\mu_{(i,j)}=\lambda_{i}\pi_{j}$. A simple calculation shows that;\\

$\overline{p}^{(n)}_{(i,j)(k,l)}=p_{ik}^{(n)}p_{jl}^{(n)}$ for fixed states $i,j,k,l$\\

As $P$ is irreducible and aperiodic, we have that $min(p_{ik}^{(n)},p_{jl}^{(n)})>0$, for sufficiently large $n$. Hence, for such $n$, $\overline{p}^{(n)}_{(i,j)(k,l)}>0$ and $\overline{P}$ is irreducible. A similar straightforward calculation gives that the distribution $\mu_{(i,j)}$ is invariant for $\overline{P}$. By well known results, this implies that $\overline{P}$ is positive recurrent. Fix a state $b$, and let $S=inf\{n\geq 1:X_{n}=Y_{n}=b\}$. Then $S$ is the first passage time in the system $\{W_{n}\}_{n\geq 0}$ to $(b,b)$, and $P(S<\infty)=1$ follows by known results, and the fact that $\overline{P}$ is irreducible and recurrent. Clearly $P(S<\infty)\leq P(T<\infty)$, so $(*)$ follows. We now calculate;\\

$P(X_{n}=j)=P(X_{n}=j,n\geq T)+P(X_{n}=j,n<T)$\\

$=P(Y_{n}=j,n\geq T)+P(X_{n}=j,n<T)$\\

by definition of $T$ and the fact that $\{X_{n}\}_{n\geq 0}$ and $\{Y_{n}\}_{n\geq 0}$ have the same transition matrix. Then;\\

$P(X_{n}=j)=P(Y_{n}=j,n\geq T)+P(Y_{n}=j,n<T)$\\

$\indent \ \ \ \ \ \ \ \ \ \ \ -P(Y_{n}=j,n<T)+P(X_{n}=j,n<T)$\\

$=P(Y_{n}=j)-P(Y_{n}=j,n<T)+P(X_{n}=j,n<T)$\\

$=\pi_{j}-P(Y_{n}=j,n<T)+P(X_{n}=j,n<T)$ $(**)$\\

We have that $P(Y_{n}=j,n<T)\leq P(n<T)$ and $P(n<T)\rightarrow P(T=\infty)=0$ as $n\rightarrow\infty$, using $(*)$. Similarly, $P(X_{n}=j,n<T)\rightarrow 0$ as $n\rightarrow\infty$. It follows that $P(X_{n}=j)\rightarrow \pi_{j}$, using $(**)$, as required. The final claim is a consequence of the fact that $p_{ij}^{(n)}=P(X_{n}=j)$ where the initial distribution of $X_{0}$ is the dirac function $\delta_{i}$.

\end{proof}

We now establish a rate of convergence result.

\begin{lemma}
\label{convergence}
Let $P$ be the transition matrix for a finite irreducible aperiodic Markov chain. Then there exists $m\geq 1$ and $\rho\in (0,1)$, such that;\\

$|p_{ij}^{(n)}-\pi_{j}|\leq (1-\rho)^{{n\over m}-1}$, for all states $i,j$\\

\end{lemma}

\begin{proof}
From Theorem \ref{coupling}, taking the initial distribution of $X_{0}$ to be $\delta_{i}$, we have that;\\

$P(X_{n}=j)=\pi_{j}-P(Y_{n}=j,n<T)+P(X_{n}=j,n<T)$\\

Hence;\\

$|p_{ij}^{(n)}-\pi_{j}|\leq P(n<T)$\\

As $P$ is irreducible and aperiodic, we have that $p_{kl}^{(n)}>0$ for all sufficiently large $n$, and all states $k,l$. As $P$ is finite, there exists an $m\geq 1$ such that $p_{kl}^{(m)}>0$ for all $k,l$. In particular, there exists $\rho\in (0,1)$ such that $p_{kl}^{(m)}>\rho$. We have that;\\

$P_{k,l}(T\leq m)\geq \sum_{u} \overline{p}^{(m)}_{(k,l),(u,u)}=\sum_{u}p_{ku}^{(m)}p_{lu}^{(m)}\geq \rho\sum_{u}p_{ku}^{(m)}=\rho$\\

$P_{k,l}(T>m)\leq (1-\rho)$\\

$P(T>m)=\sum_{(k,l)}P_{(k,l)}(T>m)\delta_{ik}\pi_{l}\leq (1-\rho)$\\

Moreover;\\

$P(T>n)\leq P(T>[{n\over m}]m)$\\

We claim that, for $k\geq 1$, $P(T>(k+1)m|T>km)\leq 1-\rho$. We have that;\\

$P(T>(k+1)m|T>km)$\\

$=\sum_{i_{km}\neq j_{km}, i_{km-1}\neq j_{km-1},\ldots,i,j}P(T>(k+1)m|W_{km}=(i_{km},j_{km}),W_{km-1}=(i_{km-1},j_{km-1}),\ldots, W_{0}=(i,j))P(W_{km}=(i_{km},j_{km}),W_{km-1}=(i_{km-1},j_{km-1}),\ldots, W_{0}=(i,j)|T>km)$\\

$=\sum_{i_{km}\neq j_{km}, i_{km-1}\neq j_{km-1},\ldots,i,j}P(T>(k+1)m|W_{km}=(i_{km},j_{km}))P(W_{km}=(i_{km},j_{km}),W_{km-1}=(i_{km-1},j_{km-1}),\ldots, W_{0}=(i,j)|T>km)$\\

$\leq (1-\rho)\sum_{i_{km}\neq j_{km}, i_{km-1}\neq j_{km-1},\ldots,i,j}P(W_{km}=(i_{km},j_{km}),W_{km-1}=(i_{km-1},j_{km-1}),\ldots, W_{0}=(i,j)|T>km)$\\

$=(1-\rho)$\\

Inductively, we have that;\\

$P(T>km)=P(T>km,T>(k-1)m)$\\

$=P(T>km|T>(k-1)m)P(T>(k-1)m)$\\

$\leq P(T>km|T>(k-1)m)(1-\rho)^{k-1}$\\

$\leq (1-\rho)^{k}$\\

It follows that $|p_{ij}^{(n)}-\pi_{j}|\leq (1-\rho)^{[{n\over m}]}\leq (1-\rho)^{{n\over m}-1}$\\

\end{proof}

\begin{lemma}
\label{transition}
Let $P$ define a Markov chain with $N$ states, $\{0,1,\ldots,N-1\}$ such that the transition probabilities of moving from state i to i-1,i,i+1 (mod N) respectively is ${1\over 3}$. Then $P$ is irreducible and aperiodic, moreover, we can choose $m=N-1$ and $\rho={1\over 3^{N-1}}$ in Lemma \ref{convergence}. It follows that;\\

$|p_{ij}^{(n)}-{1\over N}|\leq ({3^{N-1}-1\over 3^{N-1}})^{{n\over N-1}-1}=\epsilon_{n}$\\

If $N$ is odd, we can choose $m={N-1\over 2}$ and $\rho={1\over 3^{{N-1\over 2}}}$ in Lemma \ref{convergence}. It follows that;\\

$|p_{ij}^{(n)}-{1\over N}|\leq ({3^{{N-1\over 2}}-1\over 3^{{N-1\over 2}}})^{{2n\over N-1}-1}=\delta_{n}$\\

Moreover, for any initial probability distribution $\pi_{0}$, letting $\pi_{j}^{n}=P(X_{n}=j)$, we have that;\\

$|\pi_{j}^{n}-{1\over N}|\leq \epsilon_{n}$, $0\leq j\leq N-1$ $(*)$\\

and, similarly, with $\delta_{n}$ replacing $\epsilon_{n}$, when $N$ is odd. For any initial distribution $\lambda_{0}$ of positive numbers, with sum $K$, letting $\lambda_{n}=\lambda_{0}P^{n}$, we have that;\\

$|\lambda_{j}^{n}-{K\over N}|\leq K\epsilon_{n}$, $0\leq j\leq N-1$\\

and, similarly, with $\delta_{n}$ replacing $\epsilon_{n}$ when $N$ is odd.

\end{lemma}

\begin{proof}
The first two claims follow immediately by noting that for any $2$ states $k,l$ $p_{kl}^{((N-1)}\geq {1\over 3^{N-1}}$, and the transition probabilities $p_{kk}>0$, for all states $k$. This also shows that we can choose $m=N-1$ and $\rho={1\over 3^{N-1}}$. The final claim of the first part follows from Lemma \ref{convergence}. The case when $N$ is odd is similar. The penultimate claim follows by noting that $\pi_{n}=\pi_{0}P^{n}$ and calculating;\\

$\pi_{j}^{n}=\pi_{0}^{0}p_{0j}^{(n)}+\pi_{1}^{0}p_{1j}^{(n)}+\ldots+\pi_{N-1}^{0}p_{N-1,j}^{(n)}$\\

$=(\pi_{0}^{0}+\ldots+\pi_{N-1}^{0})({1\over N})+\pi_{0}^{0}\epsilon_{n}^{0}+\ldots+\pi_{N-1}^{0}\epsilon_{n}^{N-1}$\\

$={1\over N}+\epsilon_{n}'$\\

where $\epsilon_{n}^{j}\leq \epsilon_{n}$, for $0\leq j\leq N-1$ and $\epsilon_{n}'\leq\epsilon_{n}$. The case $N$ odd is the same. The final claim follows by observing that $\lambda_{n}=\lambda_{0}P^{n}=K\pi_{0}P^{n}$ for an initial probability distribution $\pi_{0}$. Then multiply the result $(*)$ through by $K$, similarly for $N$ odd.   \\

\end{proof}

\begin{lemma}
\label{nonstandard1}
Let $P$ define a non standard Markov chain with $\eta$ states, $\{0,1,\ldots,\eta-1\}$, for $\eta$ infinite, such that the transition probabilities of moving from state i to i-1,i,i+1 (mod $\eta$) respectively is ${1\over 3}$. Then, if $\epsilon$ is an infinitesimal and\\

 $n\geq (\eta-1)(1+{log(\epsilon)\over log(3^{\eta-1}-1)-log(3^{\eta-1})}$ $(*)$\\

we have for any initial probability distribution $\pi_{0}$, that;\\

$\pi_{j}^{n}\simeq{1\over\eta}$ for $0\leq j\leq \eta-1$ $(**)$\\

We obtain the same result, if $\eta$ is odd, $\epsilon$ is an infinitesimal and;\\

$n\geq {\eta-1\over 2}(1+ {log(\epsilon)\over log(3^{{\eta-1\over 2}}-1)-log(3^{{\eta-1\over 2}})})$ $(***)$\\

for any initial probability distribution $\pi_{0}$. If $\lambda_{0}$ is a nonstandard distribution with sum $\lambda$, possibly infinite, then if $\epsilon$ is an infinitesimal with $\lambda\epsilon\simeq 0$, and $n$ satisfies $(*)$, we obtain that;\\

$\pi_{j}^{n}\simeq{\lambda\over\eta}$ for $0\leq j\leq \eta-1$ $(****)$\\

and, the same result holds when $\eta$ is odd and $n$ satisfies $(***)$.

\end{lemma}

\begin{proof}
Let $Seq_{1}=\{f:\mathcal{N}\rightarrow\mathcal{R}\}$ and $Seq_{2}=\{f:\mathcal{N}^{2}\rightarrow\mathcal{R}\}$. We let;\\

\noindent $Prob_{N}=\{f\in Seq_{1}:(\forall_{m\geq N+1}f(m)=0)\wedge(\forall_{1\leq m\leq N}f(m)\geq 0)\wedge\sum_{1\leq m\leq N}f(m)=1\}$.\\

encode probability vectors of length $N$. Let $G:\mathcal{N}\rightarrow Seq_{2}$ be defined by;\\

$G(N,1,1)=G(N,1,2)=G(N,1,N)={1\over 3}$, $(\forall_{m\neq 1,2,N}G(N,1,m)=0$\\

$(\forall_{1\leq k\leq N-1}\forall_{2\leq m\leq N}G(N,k+1,1)=G(N,k,N)$, $G(N,k+1,m)=G(N,k,m-1)$\\

$\forall_{k\geq N+1}\forall_{m\geq N+1}G(N,k,m)=0$\\

$G$ encodes the transition matrices for the given Markov chain with $N$ states. Let $H:\mathcal{N}^{2}\rightarrow Seq_{2}$ be defined by;\\

$(\forall_{1\leq i,j\leq N})H(1,N,i,j)=G(N,i,j)$\\

$(\forall_{i,j>N})H(1,N,i,j)=0$\\

$(\forall_{1\leq i,j\leq N})H(n,N,i,j)=\sum_{1\leq k\leq N}H(n-1,i,k)G(N,k,j)$\\

$(\forall_{i,j>N})H(n,N,i,j)=0$\\

$H$ encodes the powers $G(N)^{n}$ of the transition matrices. We define maps $L(N,n):Prob_{N}\rightarrow Prob_{N}$ by;\\

$(\forall_{1\leq j\leq N})L(N,n)(f)(j)=\sum_{1\leq k\leq N}f(k)H(n,N,k,j)$\\

$(\forall_{j>N})L(N,n)(f)(j)=0$\\

$L(N,n)(f)$ encodes the probability vectors $\pi^{n}$ for an initial distribution $\pi_{0}$ represented by $f$.\\

By a simple rearrangement, we have that the bound in $|\pi_{j}^{n}-{1\over N}|$, from lemma \ref{transition}, can be formulated in first order logic as;\\

$\forall N\in\mathcal{N}\forall\pi\in Prob_{N}\forall\epsilon\in\mathcal{R}_{>0}\forall{n\in\mathcal{N}}(n\geq (N-1)(1+{log(\epsilon)\over log(3^{N-1}-1)-log(3^{N-1})})\rightarrow(|L(n,N)(\pi)(j)-{1\over N}|\leq\epsilon, 0\leq j\leq N-1)$\\

By transfer, we obtain a corresponding result, quantifying over ${^{*}\mathcal{N}}$. Taking $\epsilon$ to be an infinitesimal and $\eta$ to be an infinite natural number, we obtain the first result. Observe that by construction of $G,H,L$, the nonstandard Markov chain with $\eta$ states evolves by the usual nonstandard matrix multiplication by the transition matrix, of the initial probability distribution. The remaining claims are similar and left to the reader.

\end{proof}

\begin{defn}
\label{measures}
We let ${\eta,\lambda}\subset {{^{*}N}\setminus\mathcal{N}}$, and let $\nu$ satisfy the bound $(*)$ in Lemma \ref{nonstandard}, where $\epsilon$ is an infinitesimal with with $\lambda\epsilon\simeq 0$.\\

We let $\overline{\Omega_{\eta}}=\{x\in{^{*}\mathcal{R}}:0\leq x<1\}$\\

and $\overline{\mathcal{T}_{\nu}}=\{t\in{^{*}\mathcal{R}}:0\leq x\leq 1\}$\\

We let $\mathcal{C}_{\eta}$ consist of internal unions of the intervals $[{i\over\eta},{i+1\over\eta})$, for $0\leq i\leq \eta-1$, and let $\mathcal{D}_{\nu}$ consist of internal unions of $[{i\over\nu},{i+1\over\nu})$ and $\{1\}$, for $0\leq i\leq \nu-1$.\\

We define counting measures $\mu_{\eta}$ and $\lambda_{\nu}$ on $\mathcal{C}_{\eta}$ and $\mathcal{D}_{\nu}$ respectively, by setting $\mu_{\eta}([{i\over\eta},{i+1\over\eta}))={1\over\eta}$, $\lambda_{\nu}([{i\over\nu},{i+1\over\nu}))={1\over\nu}$, for $0\leq i\leq \eta-1$, $0\leq i\leq \nu-1$ respectively, and $\lambda_{\nu}(\{1\})=0$.\\

We let $(\overline{\Omega_{\eta}},\mathcal{C}_{\eta},\mu_{\eta})$ and $(\overline{\mathcal{T}}_{\nu},\mathcal{D}_{\nu},\lambda_{\nu})$ be the resulting ${*}$-finite measure spaces, in the sense of \cite{L}. We let $(\overline{\Omega_{\eta}}\times\overline{\mathcal{T}}_{\nu}, \mathcal{C}_{\eta}\times\mathcal{D}_{\nu}, \mu_{\eta}\times \lambda_{\nu})$ denote the corresponding product space.
\end{defn}

\begin{defn}
\label{initial}
Let $f:\overline{\Omega_{\eta}}\rightarrow {^{*}\mathcal{R}}_{\geq 0}$  be measurable with respect to the $*\sigma$-algebra $\mathcal{C}_{\eta}$, in the sense of \cite{L}, and suppose that ${^{*}\sum_{0\leq i\leq\eta-1}}f({i\over\eta})=\lambda$. We define $F:\overline{\Omega_{\eta}}\times \overline{\mathcal{T}_{\nu}}\rightarrow{^{*}\mathcal{R}}_{\geq 0}$ by;\\

$F({i\over\eta},{j\over\nu})=(\pi_{f}K^{j})(i)$, for $0\leq i\leq \eta-1$, $0\leq j\leq \nu$\\

$F(x,y)=F({[x\eta]\over\eta},{[y\nu]\over\nu})$, $(x,y)\in \overline{\Omega_{\eta}}\times \overline{\mathcal{T}_{\nu}}$\\

where $\pi_{f}$ is the nonstandard distribution vector corresponding to $f$, $K$ is the transition matrix of the above Markov chain with $\eta$ states, and $K^{j}$ denotes a nonstandard power.

\end{defn}

\begin{lemma}
\label{measurable}
Let $F$ be as defined in Definition \ref{initial}, then $F$ is measurable with respect to $\mathcal{C}_{\eta}\times\mathcal{D}_{\nu}$, and, moreover $F(x,1)\simeq C$ where $C=\int_{\overline{\Omega_{\eta}}}fd\mu_{\eta}$.\\
\end{lemma}

\begin{proof}
The first proposition follows by observing that the defining schema for $F$ is internal and by transfer from the result for finite measures spaces, see Lemma \label{nonstandard} for the mechanics of this transfer process. For the second proposition, observe that, by definition of the nonstandard integral, see \cite{dep}, and the assumptions on $f$, we have that $C={\lambda\over\eta}$. The result then follows by the choice of $\nu$ and the result of Lemma \label{nonstandard}.

\end{proof}

\begin{rmk}
\label{gas}
$F$ defines the evolution of a stochastic process, which we can think of as the density of a number of gas particles each moving independently and at random. This idea is made more precise in \cite{dep1}. The final density, which we refer to as the equilibrium density is close to being constant.
\end{rmk}

\begin{defn}
\label{reverse}
Let $(\overline{\Omega_{\eta}},\mathcal{E}_{\eta},\gamma_{\eta})$ be a nonstandard $*$-finite measure space. We define a reverse filtration on $\overline{\Omega_{\eta}}$ to be an internal collection of $*\sigma$-algebras $\mathcal{E}_{\eta,i}$, indexed by $0\leq i\leq\nu$, such that;\\

(i). $\mathcal{E}_{\eta,0}=\mathcal{E}_{\eta}$\\

(ii). $\mathcal{E}_{\eta,i}\subseteq \mathcal{E}_{\eta,j}$, if $i\geq j$.\\

We say that $F:\overline{\Omega_{\eta}}\times\overline{\mathcal{T}_{\nu}}\rightarrow {^{*}\mathcal{R}}$ is adapted to the filtration if $F$ is measurable with respect to $\mathcal{E}_{\eta}\times \mathcal{D}_{\nu}$ and $F_{i\over\nu}:\overline{\Omega_{\eta}}\rightarrow{^{*}\mathcal{R}}$ is measurable with respect to $\mathcal{E}_{\eta,i}$, for $0\leq i\leq \nu$.\\

If $f:\overline{\Omega_{\eta}}\rightarrow{^{*}\mathcal{R}}$ is measurable with respect to $\mathcal{E}_{\eta,j}$ and $0\leq j\leq i\leq \nu$, we define the conditional expectation $E_{\eta}(f|\mathcal{E}_{\eta,i})$ to be the unique $g:\overline{\Omega_{\eta}}\rightarrow{^{*}\mathcal{R}}$ such that $g$ is measurable with respect to $\mathcal{E}_{\eta,i}$ and;\\

$\int_{U}gd\gamma_{\eta}=\int_{U}f d\gamma_{\eta}$\\

for all $U\in \mathcal{E}_{\eta,i}$. We say that  $F:\overline{\Omega_{\eta}}\times\overline{\mathcal{T}_{\nu}}\rightarrow {^{*}\mathcal{R}}$ is a reverse martingale if;\\

(i). $F$ is adapted to the reverse filtration on $\overline{\Omega_{\eta}}$\\

(ii). $E_{\eta}(F_{j\over\nu}|\mathcal{E}_{\eta,i})=F_{i\over\nu}$ for $0\leq j\leq i\leq \nu$\\

Given $F:\overline{\Omega_{\eta}}\times\overline{\mathcal{T}_{\nu}}\rightarrow {^{*}\mathcal{R}}$ measurable with respect to $\mathcal{E}_{\eta}\times \mathcal{D}_{\nu}$, we define the cumulative density function; $P:{^{*}\mathcal{R}}\times [0,1]\rightarrow{^{*}\mathcal{R}}$ by;\\

$P(x,t)=\gamma_{\eta}(F_{t}\leq x)$\\

We say that $F_{1}$ and $F_{2}$ are equivalent in distribution, if their respective cumulative density functions $P_{1}$ and $P_{2}$ coincide.\\

\end{defn}

\begin{theorem}
\label{martingale}
Let $F$ be as in Definition \ref{initial}, then there exists a reverse filtration on $\overline{\Omega_{\eta}}$ and an extension of $F$ to $\overline{F}$ such that $\overline{F}$ is a reverse martingale. Moreover the processes $F$ and $\overline{F}$ are equivalent in distribution.

\end{theorem}

\begin{proof}
We define the reverse filtration, by setting $\mathcal{E}_{\eta,i}$ to be internal unions of the intervals $[{j\over 3^{\nu-i}\eta},{j+1\over 3^{\nu-i}\eta})$ for $0\leq j\leq 3^{\nu-i}\eta-1$, $0\leq i\leq \nu$. Clearly, this is an internal collection. It follows that $\mathcal{E}_{\eta}=\mathcal{E}_{\eta,0}$ consists of internal unions of the intervals $[{j\over 3^{\nu}\eta},{j+1\over 3^{\nu}\eta})$ for $0\leq j\leq 3^{\nu}\eta-1$, and we define the corresponding measure $\gamma_{\eta}$ by setting $\gamma_{\eta}([{j\over 3^{\nu}\eta},{j+1\over 3^{\nu}\eta}))={1\over 3^{\nu}\eta}$. Observe that $\mathcal{E}_{\eta,\nu}=\mathcal{C}_{\eta}$, the original $*\sigma$-algebra. We define special points of the first and second kind on $\overline{\Omega_{\eta}}\times \overline{\mathcal{T}_{\nu}}$ inductively, as follows;\\

Base case; the points $\{({i\over\eta},1):0\leq i\leq \eta-1\}$ are special points of the first kind.\\

Inductive hypothesis; Suppose the special points of the first and second kind on $\overline{\Omega_{\eta}}\times[{\nu-i\over\nu},1]$ have been defined for $0\leq i\leq \nu-1$\\

Then a point $(x,{\nu-i-1\over\nu})$ is special of the first kind if $(x,{\nu-i\over\nu})$ is special of the first or second kind, in addition, the points $(x+{k\over 3^{i+1}\eta},{\nu-i-1\over\nu})$, $1\leq k\leq 2$, are special of the second kind, when $(x,{\nu-i\over\nu})$ is special of the first kind.\\

We call a point of the form $({i\over\eta},{j\over\nu})$ for $0\leq i\leq \eta-1$, $0\leq j\leq\nu$ original. We associate original points to special points inductively as follows;\\

Base case; We associate the original point $({i\over\eta},1)$ to the special point $({i\over\eta},1)$, for $0\leq i\leq \eta-1$\\

Inductive hypothesis: Suppose we have associated original points to special points on $\overline{\Omega_{\eta}}\times[{\nu-i\over\nu},1]$,for $0\leq i\leq \nu-1$.\\

Then if $(x,{\nu-i-1\over\nu})$ is special of the first kind, and $({i_{x}\over\eta},{\nu-i\over\nu})$ is associated to $(x,{\nu-i\over\nu})$, we associate $({i_{x}\over\eta},{\nu-i-1\over\nu})$ to $(x,{\nu-i-1\over\nu})$. We match original points to special points of the second kind, by associating $({i_{x}-1\over\eta},{\nu-i-1\over\nu})$ to $(x+{1\over 3^{i+1}\eta},{\nu-i-1\over\nu})$ and $({i_{x}+1\over\eta},{\nu-i-1\over\nu})$ to $(x+{2\over 3^{i+1}\eta},{\nu-i-1\over\nu})$, where $(x,{\nu-i-1\over\nu})$ is special of the first kind. We adopt the convention that ${-1\over\eta}={\eta-1\over\eta}$ and ${\eta\over\eta}={0\over\eta}$.\\

We define $\overline{F}$ on $\overline{\Omega_{\eta}}\times\overline{\mathcal{T}_{\nu}}$ by setting $\overline{F}(z)=F(\Gamma(z))$, $(*)$, where $z$ is a special point of $\overline{\Omega_{\eta}}\times\overline{\mathcal{T}_{\nu}}$ and $\Gamma(z)$ is the associated original point. We then set $\overline{F}(x,y)=\overline{F}({[x3^{\nu-[\nu y]}\eta]\over 3^{\nu-[\nu y]}\eta},{[\nu y]\over \nu})$, $(**)$.\\

We claim that $\overline{F}$ is a reverse martingale. By the definition $(*)$, $\overline{F}$ is adapted to the reverse filtration, $(i)$. To verify $(ii)$ of the definition of a reverse martingale in Definition \ref{reverse}, by the tower law for conditional expectation, it is sufficient to prove that $E_{\eta}(\overline{F}_{i\over\nu}|\mathcal{E}_{i+1})=\overline{F}_{i+1\over\nu}$, for $0\leq i\leq \nu-1$. We have that;\\

$E_{\eta}(\overline{F}_{i\over\nu}|\mathcal{E}_{i+1})({j\over 3^{\nu-i-1}\eta})$\\

$=3^{\nu-i-1}\eta\int_{[{j\over 3^{\nu-i-1}\eta},{j+1\over 3^{\nu-i-1}\eta})}E_{\eta}(\overline{F}_{i\over\nu}|\mathcal{E}_{i+1})d\gamma_{\eta}$\\

$=3^{\nu-i-1}\eta\int_{[{j\over 3^{\nu-i-1}\eta},{j+1\over 3^{\nu-i-1}\eta})}\overline{F}_{i\over\nu}d\gamma_{\eta}$\\

$={3^{\nu-i-1}\eta\over 3^{\nu-i}\eta}(\sum_{k=0}^{2}\overline{F}_{i\over\nu}({3j+k\over 3^{\nu-i}\eta}))$\\

$={1\over 3}(F_{i\over\nu}({i_{x}\over\eta})+F_{i\over\nu}({i_{x}-1\over\eta})+F_{i\over\nu}({i_{x}+1\over\eta}))$\\

$=F_{i+1\over\nu}({i_{x}\over\eta})=\overline{F}_{i+1\over\nu}({j\over 3^{\nu-i-1}\eta})$\\

where $({i_{x}\over\eta},{i+1\over\nu})$ is associated to the special point $({j\over 3^{\nu-i-1}\eta},{i+1\over\nu})=(x,{i+1\over\nu})$\\

We now claim that if $({i\over\eta},{\nu-j\over\nu})$ is an original point, then it is associated to $3^{j}$ special points of the form $(y,{\nu-j\over\nu})$. We prove this by induction.\\

Base Case; $j=0$, then the original points $\{({i\over\eta},1):0\leq i\leq \eta-1\}$ are in bijection with the same special points.\\

Inductive hypothesis; Suppose that, for $0\leq i\leq \eta-1$, each original point $({i\over\eta},{\nu-j\over\nu})$ is associated to $3^{j}$ special points of the form $(y,{\nu-j\over\nu})$.\\

If the original point $({i\over\eta},{\nu-j-1\over\nu})$ is associated to a special point $(x,{\nu-j-1\over\nu})$, there are $3$ cases, either $({i\over\eta},{\nu-j\over\nu})$ is associated to $(x,{\nu-j\over\nu})$, or, $({i-1\over\eta},{\nu-j\over\nu})$ is associated to $(x-{1\over 3^{j+1}\eta},{\nu-j\over\nu})$, or $({i+1\over\eta},{\nu-j\over\nu})$ is associated to $(x-{2\over 3^{j+1}\eta},{\nu-j\over\nu})$. These cases are disjoint and all occur, so we obtain a total of $3.3^{j}=3^{j+1}$ assignments.\\

Using this result and the definition of $\overline{F}$, for a given $\alpha\in{^{*}\mathcal{R}}$;\\

 $3^{\nu-[\nu t]}{^{*}Card}(\{i:0\leq i\leq \eta-1,F({i\over\eta},{[\nu t]\over \nu})=\alpha\})$\\

 $={^{*}Card}(\{i:0\leq i\leq 3^{\nu-[\nu t]}\eta-1,\overline{F}({i\over 3^{\nu-[\nu t]}\eta},{[\nu t]\over \nu})=\alpha\})$\\

 We now calculate;\\

$\gamma_{\eta}(F_{t}=\alpha)$\\

$={1\over\eta}{^{*}Card}(\{i:0\leq i\leq \eta-1,F({i\over\eta},{[\nu t]\over \nu})=\alpha\})$\\

$={1\over 3^{\nu-[\nu t]}\eta}{^{*}Card}(\{i:0\leq i\leq 3^{\nu-[\nu t]}\eta-1,\overline{F}({i\over 3^{\nu-[\nu t]}\eta},{[\nu t]\over \nu})=\alpha\})$\\

$=\gamma_{\eta}(\overline{F}_{t}=\alpha)$\\

Using the measurability of $F_{[\nu t]\over \nu}$ and $\overline{F}_{[\nu t]\over \nu}$ with respect to the algebras $\mathcal{C}_{\eta}$ and $\mathcal{E}_{\eta,\nu-[\nu t]}$ respectively.

\end{proof}

\begin{rmk}
The advantage of working with a reverse martingale to analyse the cumulative density function of $F$ is that we have available a nonstandard martingale representation theorem, Ito's Lemma and a strategy to obtain a Fokker-Planck type equation. This is work in progress, see \cite{dep2}.

\end{rmk}

We make some considerations in connection with the heat equation.

\begin{defn}
\label{smooth}
We let $S^{1}(1)$ denote the circle of radius $1$, which we identify with the closed interval $[-\pi,\pi]$, via $\mu:[-\pi,\pi]\rightarrow S^{1}(1)$, $\mu(\theta)=e^{i\theta}$. We let $C([-\pi,\pi])=\{\mu^{*}(g):g\in C^{\infty}(S^{1})\}$ and $C^{\infty}([-\pi,\pi])=\{\mu^{*}(g):g\in C^{\infty}(S^{1})\}$. We let $T=[-\pi,\pi]\times\mathcal{R}_{\geq 0}$ and $T^{0}=(-\pi,\pi)\times\mathcal{R}_{>0}$ denote its interior. We let $C^{\infty}(T)=\{G\in C(T):G_{t}\in C^{\infty}([-\pi,\pi])$, for $t\in\mathcal{R}_{\geq 0}, G|T^{0}\in C^{\infty}(T^{0})\}$. If $h\in C([-\pi,pi])$, we define its Fourier transform by;\\

$\mathcal{F}(h)(m)={1\over 2\pi}\int_{-\pi}^{\pi}h(x)e^{-imx}dx$\\

If $g\in C(T)$, we define its Fourier transform in space by;\\

 $\mathcal{F}(g)(m,t)={1\over 2\pi}\int_{-\pi}^{\pi}g(x,t)e^{-imx}dx$
\end{defn}

\begin{lemma}
\label{heatequation}
If $g\in C^{\infty}([-\pi,\pi])$, there exists a unique $G\in C^{\infty}(T)$, with $G_{0}=g$, such that $G$ satisfies the heat equation;\\

${\partial{G}\over \partial t}={\partial^{2}G\over \partial x^{2}}$ $(*)$\\

on $T^{0}$.\\

\end{lemma}

\begin{proof}
Suppose, first, there exists such a solution $G$, then, applying $\mathcal{F}$ to $(*)$, we must have that;\\

$\mathcal{F}({\partial{G}\over \partial t}-{\partial^{2}G\over \partial x^{2}})(m,t)=0$ $(t>0,m\in\mathcal{Z})$\\

Differentiating under the integral sign, we have that;\\

 $\mathcal{F}({\partial{G}\over \partial t})={\partial \mathcal{F}(G)\over \partial t}(m,t)$, for $t>0,m\in\mathcal{Z}$\\

 Integrating by parts and using the fact that $G_{t}\in C^{\infty}([-\pi,\pi])$, for $t>0$, we have that;\\

$\mathcal{F}{\partial^{2}G\over \partial x^{2}}=-m^{2}\mathcal{F}(G)(m,t)$, for $t>0,m\in\mathcal{Z}$\\

We thus obtain the sequence of ordinary differential equations, indexed by $m\in\mathcal{Z}$;\\

${\partial \mathcal{F}(G)\over \partial t}+m^{2}\mathcal{F}(G)(m,t)=0$ $(t>0)$\\

with initial condition, given by;\\

$\mathcal{F}(G)(m,0)=\mathcal{F}(g)(m)$\\

By Picard's Theorem, this has the unique solution, given by;\\

$\mathcal{F}(G)(m,t)=e^{-m^{2}t}\mathcal{F}(g)(m)$ $(t\geq 0)$\\

As $G_{t}\in C^{\infty}([-\pi,\pi])$, its Fourier series converges absolutely to $G_{t}$ and, in particular, $G_{t}$ is determined by its Fourier coefficients, for $t>0$. It follows that $G$ is a unique solution.\\

If $g\in C^{\infty}([-\pi,\pi])$, its Fourier series converges absolutely to $g$, hence, the series;\\

$\sum_{m\in\mathcal{Z}}e^{-m^{2}t}\mathcal{F}(g)(m)e^{imx}$\\

are absolutely convergent for $t>0$. It follows that $G$ defined by;\\

$G(x,t)=\sum_{m\in\mathcal{Z}}e^{-m^{2}t}\mathcal{F}(g)(m)e^{imx}$\\

is a solution of the required form.
\end{proof}

\begin{rmk}
Observe that as $t\rightarrow\infty$, $G(x,t)\rightarrow \mathcal{F}(g)(0)={1\over 2\pi}\int_{-\pi}^{\pi}gdx$, a phenomenon we observed in Lemma \ref{measurable}. This suggests that the two processes are connected, we make this observation more precise below.

\end{rmk}

\begin{defn}
If $\eta\in{{^{*}\mathcal{N}}\setminus\mathcal{N}}$, we let $\overline{\mathcal{V}_{\eta}}={^{*}\bigcup_{0\leq i\leq 2\eta-1}}[-\pi+\pi{i\over\eta},-\pi+\pi{i+1\over\eta})$, so that $\overline{\mathcal{V}_{\eta}}={^{*}[-\pi,\pi)}$. We let $\mathcal{D}_{\eta}$ denote the associated $*$-finite algebra, generated by the intervals $[-\pi+\pi{i\over\eta},-\pi+\pi{i+1\over\eta})$, for $0\leq i\leq 2\eta-1$, and $\mu_{\eta}$ the associated counting measure defined by $\mu_{\eta}([-\pi+\pi{i\over\eta},-\pi+\pi{i+1\over\eta}))={\pi\over\eta}$. We let $(\overline{\mathcal{V}_{\eta}},L(\mathcal{D}_{\eta}),L(\mu_{\eta}))$ denote the associated Loeb space, see.... If $\nu\in{{^{*}\mathcal{N}}\setminus\mathcal{N}}$, we let $\overline{\mathcal{T}_{\nu}}={^{*}\bigcup_{0\leq i\leq \nu^{2}-1}}[{i\over\nu},{i+1\over\nu})$, so that $\overline{\mathcal{T}_{\nu}}=[0,\nu)\subset{^{*}\mathcal{R}_{\geq 0}}$.We let $\mathcal{C}_{\eta}$ denote the associated $*$-finite algebra, generated by the intervals $[{i\over\nu},{i+1\over\nu})$, for $0\leq i\leq \nu^{2}-1$, and $\lambda_{\nu}$ the associated counting measure defined by $\lambda_{\nu}([{i\over\nu},{i+1\over\nu}))={1\over\nu}$. We let $(\overline{\mathcal{T}_{\nu}},L(\mathcal{C}_{\nu}),L(\lambda_{\nu}))$ denote the associated Loeb space.\\

We let $([-\pi,\pi],\mathfrak{D},\mu)$ denote the interval $[-\pi,\pi]$, with the completion $\mathfrak{D}$ of the Borel field, and $\mu$ the restriction of Lebesgue measure. We let $(\mathcal{R}_{\geq 0}\cup\{+\infty\},\mathfrak{C},\lambda)$ denote the extended real half line, with the completion $\mathfrak{C}$ of the extended Borel field, and $\lambda$ the extension of Lebesgue measure, with $\lambda(+\infty)=\infty$, see....\\

We let $(\overline{\mathcal{V}_{\eta}}\times \overline{\mathcal{T}_{\nu}},\mathcal{D}_{\eta}\times\mathcal{C}_{\nu},\mu_{\eta}\times\lambda_{\nu})$ be the associated product space and $(\overline{\mathcal{V}_{\eta}}\times \overline{\mathcal{T}_{\nu}},L(\mathcal{D}_{\eta}\times\mathcal{C}_{\eta}),  L(\mu_{\eta}\times \lambda_{\nu}))$ be the corresponding Loeb space. $(\overline{\mathcal{V}_{\eta}}\times \overline{\mathcal{T}_{\nu}},L(\mathcal{D}_{\eta})\times L(\mathcal{C}_{\nu}),L(\mu_{\eta})\times L(\lambda_{\nu}))$ is the complete product of the Loeb spaces $(\overline{\mathcal{V}_{\eta}},L(\mathcal{D}_{\eta}),L(\mu_{\eta}))$ and $(\overline{\mathcal{T}_{\nu}},L(\mathcal{C}_{\nu}),L(\lambda_{\nu}))$. Similarly, $([-\pi,\pi]\times(\mathcal{R}_{\geq 0}\cup\{+\infty\} ,\mathfrak{D}\times\mathfrak{C},\mu\times\lambda)$ is the complete product of $([-\pi,\pi],\mathfrak{D},\mu)$ and $(\mathcal{R}_{\geq 0}\cup\{+\infty\},\mathfrak{C},\lambda)$.\\

We let $({^{*}\mathcal{R}},{^{*}\mathfrak{E}})$ denote the hyperreals, with the transfer of the Borel field $\mathfrak{C}$ on $\mathcal{R}$. A function $f:(\overline{\mathcal{V}_{\eta}},\mathcal{D}_{\eta})\rightarrow ({^{*}\mathcal{R}},{^{*}\mathfrak{E}})$ is measurable, if $f^{-1}:{^{*}\mathfrak{E}}\rightarrow \mathcal{D}_{\eta}$. The same definition holds for ${\mathcal{T}_{\nu}}$.  Similarly, $f:(\overline{\mathcal{V}_{\eta}}\times\overline{\mathcal{T}_{\nu}},\mathcal{D}_{\eta}\times\mathcal{C}_{\nu})\rightarrow ({^{*}\mathcal{R}},{^{*}\mathfrak{E}})$ is measurable, if $f^{-1}:{^{*}\mathfrak{E}}\rightarrow \mathcal{D}_{\eta}\times\mathfrak{C}_{\nu}$. Observe that this is equivalent to the definition given in \cite{Loeb}. We will abbreviate this notation to $f:\overline{\mathcal{V}_{\eta}}\rightarrow{^{*}\mathcal{R}}$, $f:\overline{\mathcal{V}_{\eta}}\rightarrow{^{*}\mathcal{R}}$ or $f:\overline{\mathcal{V}_{\eta}}\times\overline{\mathcal{T}_{\nu}}\rightarrow{^{*}\mathcal{R}}$ is measurable, $(*)$. The same
applies to $({^{*}\mathcal{C}},{^{*}\mathfrak{E}})$, the hyper complex numbers, with the transfer of the Borel field $\mathfrak{E}$, generated by the complex topology. Observe that $f:\overline{\mathcal{V}_{\eta}}\rightarrow{^{*}\mathcal{C}}$, $f:\overline{\mathcal{T}_{\nu}}\rightarrow{^{*}\mathcal{C}}$  $f:\overline{\mathcal{V}_{\eta}}\times\overline{\mathcal{T}_{\nu}}\rightarrow{^{*}\mathcal{C}}$ is measurable, in this sense, iff $Re(f)$ and $Im(f)$ are measurable in the sense of $(*)$.\\

We let $\overline{\mathcal{S}_{\eta,\nu}}=\overline{\mathcal{V}_{\eta}}\times\overline{\mathcal{T}_{\nu}}$ and;\\

$V(\overline{\mathcal{V}_{\eta}})=\{f:\overline{\mathcal{V}_{\eta}}\rightarrow {^{*}}{\mathcal{C}},\ f\ measurable\ d(\mu_{\eta})\}$\\

and, similarly, we define $V(\overline{\mathcal{T}_{\nu}})$. Let;\\

$V(\overline{\mathcal{S}_{\eta,\nu}})=\{f:\overline{\mathcal{S}_{\eta,\nu}}\rightarrow {^{*}}{\mathcal{C}},\ f\ measurable\ d(\mu_{\eta}\times\lambda_{\nu})\}$\\

\end{defn}

\begin{lemma}
\label{squaremmp}

The identity;\\

$i:(\overline{\mathcal{V}_{\eta}}\times \overline{\mathcal{T}_{\nu}}, L(\mathcal{D}_{\eta}\times \mathcal{C}_{\nu}),L(\mu_{\eta}\times \lambda_{\nu}) )$\\

$\rightarrow (\overline{\mathcal{V}_{\eta}}\times \overline{\mathcal{T}_{\nu}}, L(\mathcal{D}_{\eta})\times L(\mathcal{C}_{\nu}), L(\mu_{\eta})\times L(\lambda_{\nu}))$\\

and the standard part mapping;\\

$st:(\overline{\mathcal{V}_{\eta}}\times \overline{\mathcal{T}_{\nu}}, L(\mathcal{D}_{\eta})\times L(\mathcal{C}_{\nu}), L(\mu_{\eta})\times L(\lambda_{\nu}))\rightarrow [-\pi,\pi]\times \mathcal{R}_{\geq 0}\cup\{+\infty\} $\\

are measurable and measure preserving.
\end{lemma}
\begin{proof}
The proof is similar to Lemma 0.2 in \cite{dep3}, using Caratheodory's Extension Theorem and Theorem 22 of \cite{and}.\\
\end{proof}

\begin{defn}{Discrete Partial Derivatives}\\
\label{partials}

Let $f:\overline{\mathcal{V}_{\eta}}\rightarrow{^{*}\mathcal{C}}$ be measurable. As in \cite{dep1}, we define the discrete derivative $f'$ to be the unique measurable function satisfying;\\

$f'(-\pi+\pi{i\over\eta})={\eta\over2\pi}(f(-\pi+\pi{i+1\over\eta})-f(-\pi+\pi{i-1\over\eta}))$;\\

for $i\in{^{*}\mathcal{N}}_{1\leq i\leq 2\eta-2}$.\\

$f'(\pi-{\pi\over\eta})={\eta\over 2\pi}(f(-\pi)-f(\pi-\pi{2\over\eta}))$\\

$f'(-\pi)={\eta\over 2\pi}(f(-\pi+{\pi\over\eta})-f(\pi-{\pi\over\eta}))$\\

Let $f:\overline{\mathcal{T}_{\nu}}\rightarrow{^{*}\mathcal{C}}$ be measurable. As in \cite{dep1}, we define the discrete derivative $f'$ to be the unique measurable function satisfying;\\

$f'({i\over\nu})=\nu(f({i+1\over\nu})-f({i\over\nu}))$;\\

for $i\in{^{*}\mathcal{N}}_{0\leq i\leq \nu^{2}-2}$.\\

$f'({\nu-1\over\nu})=0$;\\

If $f:\overline{\mathcal{V}_{\eta}}\rightarrow{^{*}\mathcal{C}}$ is measurable, then we define the shift (left, right);\\

$f^{lsh}(-\pi+\pi{j\over \eta})=f(-\pi+\pi{j+1\over\eta})$ for $0\leq j\leq 2\eta-2$\\

$f^{lsh}(\eta-{\pi\over \eta})=f(-\pi)$\\

$f^{rsh}(-\pi+\pi{j\over\eta})=f(-\pi+\pi{j-1\over\eta})$ for $1\leq j\leq 2\eta-1$\\

$f^{rsh}(-\pi)=f(\pi-{\pi\over\eta})$\\

If $f:\overline{\mathcal{T}_{\nu}}\rightarrow{^{*}\mathcal{C}}$ is measurable, then we define the shift (left, right);\\

$f^{lsh}({j\over\nu})=f({j+1\over\nu})$ for $0\leq j\leq \nu^{2}-2$\\

$f^{lsh}(\nu-{1\over\nu})=f(0)$\\

$f^{rsh}({j\over\nu})=f({j-1\over\nu})$ for $1\leq j\leq \nu^{2}-1$\\

$f^{rsh}(0)=f(\nu-{1\over\nu})$\\

If $f:\overline{\mathcal{V}_{\eta}}\times\overline{\mathcal{T}_{\nu}}\rightarrow{^{*}\mathcal{C}}$ is measurable. Then we define $\{{\partial f\over\partial x},{\partial f\over\partial t}\}$ to be the unique measurable functions satisfying;\\

${\partial f\over\partial x}(-\pi+\pi{i\over\eta},t)={\eta\over 2\pi}(f(-\pi+\pi{i+1\over\eta},t)-f(-\pi+\pi{i-1\over\eta},y))$;\\

for $i\in{^{*}\mathcal{N}}_{1\leq i\leq 2\eta-2}, t\in\overline{\mathcal{T}_{\nu}}$\\

${\partial f\over\partial x}(\pi-{\pi\over\eta},t)={\eta\over 2\pi}(f(-\pi,t)-f(\pi-\pi{2\over\eta},t))$\\

${\partial f\over\partial x}(-\pi,t)={\eta\over 2\pi}(f(-\pi+{\pi\over\eta},t)-f(\pi-{\pi\over\eta},t))$\\

${\partial f\over\partial t}(x,{j\over\nu})=(f(x,{j+1\over\nu})-f(x,{j\over\nu}))$;\\

for $j\in{^{*}\mathcal{N}}_{0\leq j\leq \nu^{2}-2}, x\in\overline{\mathcal{H}_{\eta}}$\\

${\partial f\over\partial t}(x,\nu-{1\over\nu})=0$\\

We define $\{f^{lsh_{x}},f^{lsh_{t}},f^{rsh_{x}},f^{rsh_{t}}\}$ by;\\

$f^{lsh_{x}}(x_{0},t_{0})=(f_{t})^{lsh}(x_{0})$\\

$f^{lsh_{t}}(x_{0},t_{0})=(f_{x})^{lsh}(t_{0})$\\

$f^{rsh_{x}}(x_{0},t_{0})=(f_{t})^{rsh}(x_{0})$\\

$f^{rsh_{t}}(x_{0},t_{0})=(f_{x})^{rsh}(t_{0})$\\

where, if $(x_{0},t_{0})\in\overline{\mathcal{V}_{\eta}}\times\overline{\mathcal{T}_{\nu}}$;\\

$f_{t_{0}}(x_{0})=f_{x_{0}}(t_{0})=f(\pi{[{\eta x_{0}\over \pi}]\over\eta},{\nu t_{0}\over \nu})$\\

\end{defn}

\begin{rmk}
\label{rmkpartials}
If $f$ is measurable, then so are;\\

 $\{{\partial f\over\partial x},{\partial f\over\partial t},{\partial^{2} f\over\partial x^{2}},f_{x},f_{t},f^{lsh_{x}},f^{lsh_{t}},f^{rsh_{x}},f^{rsh_{t}},f^{sh_{x}^{2}},f^{sh_{t}^{2}},f^{rsh_{x}^{2}},f^{rsh_{t}^{2}}\}$\\

 This follows immediately, by transfer, from the corresponding result for the discrete derivatives and shifts of discrete functions $f:{\mathcal{H}_{n}}\times{\mathcal{T}_{m}}\rightarrow\mathcal{C}$, where $n,m\in{\mathcal{N}}$, see Definition 0.15 and Definition 0.18 of \cite{dep4}.
\end{rmk}

\begin{lemma}
\label{discretederivativeshift}

Let $g,h:\overline{\mathcal{V}_{\eta}}\rightarrow{^{*}\mathcal{C}}$ be measurable. Then;\\

$(i)$. $\int_{\overline{\mathcal{V}_{\eta}}}g'(y)d\mu_{\eta}(y)=0$\\

$(ii)$. $(gh)'=g'h^{lsh}+g^{rsh}h'$\\

$(iii)$. $\int_{\overline{\mathcal{V}_{\eta}}}(g'h)(y)d\mu_{\eta}(y)=-\int_{\overline{\mathcal{V}_{\eta}}}gh'd\mu_{\eta}(y)$\\

$(iv)$. $\int_{\overline{\mathcal{V}_{\eta}}}g(y)d\mu_{\eta}(y)=\int_{\overline{\mathcal{V}_{\eta}}}g^{lsh}(y)d\mu_{\eta}(y)=\int_{\overline{\mathcal{V}_{\eta}}}g^{rsh}(y)d\mu_{\eta}(y)$\\

$(v)$. $(g')^{rsh}=(g^{rsh})'$, $(g')^{lsh}=(g^{lsh})'$ \\

$(vi)$. $\int_{\overline{\mathcal{V}_{\eta}}}(g''h)(y)d\mu_{\eta}(y)=\int_{\overline{\mathcal{V}_{\eta}}}(gh'')(y)d\mu_{\eta}(y)$\\

 (\footnote{Similar results hold for $\{sh_{x},{\partial \over \partial x},{\partial \over \partial t}\}$. Namely, if $g,h:\overline{\mathcal{S}_{\eta,\nu}}\rightarrow{^{*}\mathcal{C}}$ are measurable. Then;\\

$(i)$. $\int_{\overline{\mathcal{S}_{\eta,\nu}}}{\partial g\over \partial x}d(\mu_{\eta}\times\lambda_{\nu})=0$\\

$(ii)$. ${\partial gh\over \partial x}={\partial g\over \partial x}h^{lsh_{x}}+g^{rsh_{x}}{\partial h\over \partial x}$\\

$(iii)$. $\int_{\overline{\mathcal{S}_{\eta,\nu}}}{\partial g\over \partial x}hd(\mu_{\eta}\times\lambda_{\nu})=-\int_{\overline{\mathcal{S}_{\eta,\nu}}}g{\partial h\over \partial x}d(\mu_{\eta}\times\lambda_{\nu})$\\

$(iv)$. $\int_{\overline{\mathcal{S}_{\eta,\nu}}}gd(\mu_{\eta}\times\lambda_{\nu})=\int_{\overline{\mathcal{S}_{\eta,\nu}}}g^{lsh_{x}}d(\mu_{\eta}\times\lambda_{\nu})=\int_{\overline{\mathcal{S}_{\eta,\nu}}}g^{rsh_{x}}d(\mu_{\eta}\times\lambda_{\nu})$\\

$(v)$. $({\partial g\over \partial x})^{lsh_{x}}={\partial (g^{lsh_{x}})\over \partial x}$, and, similarly, with $rsh_{x}$ replacing $lsh_{x}$. \\

$(vi)$. $\int_{\overline{\mathcal{S}_{\eta,\nu}}}({\partial^{2} g\over \partial x^{2}}h)d(\mu_{\eta}\times\lambda_{\nu})=\int_{\overline{\mathcal{S}_{\eta,\nu}}}(g{\partial^{2} h\over \partial x^{2}})d(\mu_{\eta}\times\lambda_{\nu})$ $(*)$\\

For $(i)$, using $(i)$ from the argument in the main proof, we have;\\

$\int_{\overline{\mathcal{S}_{\eta,\nu}}}{\partial g\over \partial x}d(\mu_{\eta}\times\lambda_{\nu})$\\

$=\int_{\overline{\mathcal{V}_{\eta}}}(\int_{\overline{\mathcal{T}_{\nu}}}({\partial g\over \partial x})_{t}d\mu_{\eta})d\lambda_{\nu}(t)$\\

$=\int_{\overline{\mathcal{V}_{\eta}}}(\int_{\overline{\mathcal{T}_{\nu}}}({\partial g_{t}\over \partial x})d\mu_{\eta})d\lambda_{\nu}(t)$\\

$=\int_{\overline{\mathcal{T}_{\nu}}}0d\lambda_{\nu}(t)=0$\\

The proofs of $(ii),(iii),(iv)$ are similar to the main proof, relying on the result of $(i)$. $(v)$ follows easily from Definitions \ref{partials} and $(vi)$ follows, repeating the result of $(iii)$, and applying $(v)$.\\

})

\end{lemma}

\begin{proof}
 In the first part, for $(i)$, we have, using Definition \ref{partials}, that;\\

$\int_{\overline{\mathcal{V}_{\eta}}}g'(y)d\mu_{\eta}(y)$\\

$={\pi\over\eta}[{^{*}}\sum_{1\leq j\leq 2\eta-2}{\eta\over 2\pi}[g(-\pi+\pi({j+1\over\eta}))-g(-\pi+\pi({j-1\over\eta}))]$\\

$+{\eta\over 2\pi}[g(-\pi+{\pi\over\eta})-g(\pi-{\pi\over\eta})]+{\eta\over 2\pi}[g(-\pi)-g(\pi-2{\pi\over\eta})]]=0$\\

For $(ii)$, we calculate;\\

$(gh)'(-\pi+\pi{j\over\eta})=$\\

$={\eta\over 2\pi}(gh(-\pi+\pi{j+1\over\eta})-gh(-\pi+\pi{j-1\over\eta}))$\\

$={\eta\over 2\pi}(gh(-\pi+\pi{j+1\over\eta})-g(-\pi+\pi{j-1\over\eta})h(-\pi+\pi{j+1\over\eta})$\\

$+g(-\pi+\pi{j-1\over\eta})h(-\pi+\pi{j+1\over\eta})-gh(-\pi+\pi{j-1\over\eta}))$\\

$=g'(-\pi+\pi{j\over\eta})h(-\pi+\pi{j+1\over\eta})+g(-\pi+\pi{j-1\over\eta})h'(-\pi+\pi{j\over\eta})$\\

$=(g'h^{lsh}+g^{rsh}h')(-\pi+\pi{j\over \eta})$\\

Combining $(i),(ii)$, we have;\\

$0=\int_{\overline{\mathcal{V}_{\eta}}}(gh)'(x)d\mu_{\eta}(x)$\\

$=\int_{\overline{\mathcal{V}_{\eta}}}(g'h^{lsh}+g^{rsh}h')(x)d\mu_{\eta}(x)$\\

and, rearranging, that;\\

$\int_{\overline{\mathcal{V}_{\eta}}}(g'h^{lsh})d\mu_{\eta}=-\int_{\overline{\mathcal{H}_{\eta}}}(g^{rsh}h')d\mu_{\eta}$\\

For $(iv)$, we have that;\\

$\int_{\overline{\mathcal{V}_{\eta}}}g^{rsh}(y)d\mu_{\eta}(y)$\\

$={\pi\over \eta}({^{*}\sum}_{0\leq j\leq 2\eta-1}g^{rsh}(-\pi+\pi{j\over\eta}))$\\

$={\pi\over\eta}({^{*}\sum}_{1\leq j\leq 2\eta-2}g(-\pi+\pi{j-1\over\eta})+g(\pi-{\pi\over\eta}))$\\

$={\pi\over\eta}({^{*}\sum}_{0\leq j\leq 2\eta-1}g(-\pi+\pi{j\over\eta})$\\

$=\int_{\overline{\mathcal{V}_{\eta}}}g(y)d\mu_{\eta}(y)$\\

A similar calculation holds with $g^{lsh}$. For $(v)$, we have for $2\leq j\leq 2\eta-2$;\\

$(g')^{rsh}(-\pi+\pi{j\over\eta})$\\

$=g'(-\pi+\pi{j-1\over \eta})$\\

$={\eta\over 2\pi}(g(-\pi+\pi{j\over \eta})-g(-\pi+\pi{j-2\over \eta}))$\\

$(g^{rsh})'(-\pi+\pi{j\over\eta})$\\

$={\eta\over 2\pi}(g^{rsh}(-\pi+\pi{j+1\over \eta})-g^{rsh}(-\pi+\pi{j-1\over \eta}))$\\

$={\eta\over 2\pi}(g(-\pi+\pi{j\over \eta})-g(-\pi+\pi{j-2\over \eta}))$\\

Similar calculations hold for the remaining $j$ to give that $(g')^{rsh}=(g^{rsh})'$, and the calculation $(g')^{lsh}=(g^{lsh})'$ is also similar.\\

It follows that;\\

$\int_{\overline{\mathcal{V}_{\eta}}}(g'h)d\mu_{\eta}$\\

$=\int_{\overline{\mathcal{V}_{\eta}}}(g'(h^{rsh})^{lsh})d\mu_{\eta}$\\

$=-\int_{\overline{\mathcal{V}_{\eta}}}(g^{rsh}(h^{rsh})')d\mu_{\eta}$\\

$=-\int_{\overline{\mathcal{V}_{\eta}}}(g^{rsh}(h')^{rsh})d\mu_{\eta}$\\

$=-\int_{\overline{\mathcal{V}_{\eta}}}(gh'))d\mu_{\eta}$\\

which gives $(vi)$, using $(iv),(v)$.\\

\end{proof}

\begin{defn}
\label{restriction}
If $\eta$ is even, we define a restriction $\overline{()}:\overline{\mathcal{V}_{\eta}}\rightarrow \overline{\mathcal{V}_{{\eta\over 2}}}$. Namely;\\

$\overline{f}(-\pi+\pi{2i\over\eta})=f(-\pi+\pi{2i\over\eta})$;\\

for $i\in{^{*}\mathcal{N}}_{0\leq i\leq \eta-1}$.\\
\end{defn}

\begin{lemma}
\label{restrictionderivative}
Let notation be as in Definitions \ref{restriction} and \ref{partials}, then;\\

$\overline{f'}(-\pi+\pi{2 i\over\eta})={\eta\over2\pi}(f(-\pi+\pi{2 i+1\over\eta})-f(-\pi+\pi{2 i-1\over\eta}))$;\\

for $i\in{^{*}\mathcal{N}}_{1\leq i\leq \eta-1}$.\\

$\overline{f'}(-\pi)={\eta\over 2\pi}(f(-\pi+{\pi\over\eta})-f(\pi-{\pi\over\eta}))$\\

and;\\

$\overline{f^{lsh}}(-\pi+\pi{2 j\over \eta})=f(-\pi+\pi{2j+1\over\eta})$ for $0\leq j\leq \eta-1$\\

$\overline{f^{rsh}}(-\pi+\pi{2 j\over\eta})=f(-\pi+\pi{2 j-1\over\eta})$ for $1\leq j\leq \eta-1$\\

$\overline{f^{rsh}}(-\pi)=f(\pi-{\pi\over\eta})$\\

\end{lemma}

\begin{proof}
The proof is an immediate consequence of Definitions \ref{restriction} and \ref{partials}

\end{proof}

\begin{rmk}
It is important to note that, in general $\overline{f'}\neq \overline{f}'$ and, similarly, for $lsh, rsh$.

\end{rmk}

\begin{lemma}
Let $\{g,h\}\subset V(\overline{\mathcal{V}_{\eta}})$ be measurable, then;\\

$(i)$. $\int_{\overline{\mathcal{V}_{{\eta\over 2}}}}\overline{g'}(y)d\mu_{{\eta\over 2}}(y)=0$\\

$(ii)$. $\overline{(gh)'}=\overline{g'}\overline{h^{lsh}}+\overline{g^{rsh}}\overline{h'}$\\

$(iii)$. $\int_{\overline{\mathcal{V}_{{\eta\over 2}}}}\overline{(g'h)}(y)d\mu_{{\eta\over 2}}(y)=-\int_{\overline{\mathcal{V}_{{\eta\over 2}}}}\overline{g^{rsh}(h')^{rsh}}d\mu_{\eta}(y)$\\

\end{lemma}
\begin{proof}
For $(i)$, we have that;\\

 $\int_{\overline{\mathcal{V}_{{\eta\over 2}}}}\overline{g'}(y)d\mu_{{\eta\over 2}}(y)$\\

$={2\pi\over\eta}[{^{*}}\sum_{1\leq j\leq \eta-1}{\eta\over 2\pi}[g(-\pi+\pi({2j+1\over\eta}))-g(-\pi+\pi({2j-1\over\eta}))]$\\

$+{\eta\over 2\pi}[g(-\pi+{\pi\over\eta})-g(\pi-{\pi\over\eta})]]=0$\\

$(ii)$ is clear from the main proof and taking restrictions.\\

For $(iii)$, integrating both sides of $(ii)$ and using $(i)$, we have that;\\

$\int_{\overline{\mathcal{V}_{{\eta\over 2}}}}\overline{g'h^{lsh}}d\mu_{{\eta\over 2}}(y)=-\int_{\overline{\mathcal{V}_{{\eta\over 2}}}}\overline{g^{rsh}h'}d\mu_{{\eta\over 2}}(y)$ $(*)$\\

Then;\\

$\int_{\overline{\mathcal{V}_{{\eta\over 2}}}}\overline{g'h}d\mu_{{\eta\over 2}}(y)$\\

$=\int_{\overline{\mathcal{V}_{{\eta\over 2}}}}\overline{g'(h^{rsh})^{lsh}}d\mu_{{\eta\over 2}}(y)$\\

$=-\int_{\overline{\mathcal{V}_{{\eta\over 2}}}}\overline{g^{rsh}(h^{rsh})'}d\mu_{{\eta\over 2}}(y)$ by $(*)$\\

$=-\int_{\overline{\mathcal{V}_{{\eta\over 2}}}}\overline{g^{rsh}(h')^{rsh}}d\mu_{{\eta\over 2}}(y)$ by the main proof\\

\end{proof}

\begin{lemma}
\label{nsheat}
Given a measurable boundary conditions $f\in V(\overline{\mathcal{V}_{\eta}})$, there exists a unique measurable $F\in V(\overline{\mathcal{S}_{\eta,\nu}})$, satisfying the nonstandard heat equation;\\

${\partial F\over\partial t}={\partial^{2}f\over\partial x^{2}}$\\

on $({\overline{\mathcal{T}_{\nu}}\setminus [\nu-{1\over\nu},\nu)})\times\overline{\mathcal{V}_{\eta}}$\\

with $F(0,x)=f(x)$, for $x\in\overline{\mathcal{V}_{\eta}}$, $(*)$.\\

Moreover, if $\eta\leq\sqrt{2\pi}\nu$, and, there exists $M\in\mathcal{R}$, with $max\{f,f',f''\}\leq M$, then $max\{F,{\partial F\over\partial x},{\partial^{2} F\over\partial x^{2}}\}\leq M$.

\end{lemma}

\begin{proof}
Observe that, by Definition \ref{partials}, if $F:\overline{\mathcal{S}_{\eta,\nu}}\rightarrow{^{*}\mathcal{C}}$ is measurable, then;\\

${\partial^{2}F\over\partial x^{2}}(-\pi+\pi{i\over\eta},t)={\eta^{2}\over 4\pi^{2}}(F(-\pi+\pi{i+2\over\eta},t)-2F(-\pi+\pi{i\over\eta},t)+F(-\pi+\pi{i-2\over\eta},t))$\\

$(2\leq i\leq 2\eta-3), t\in\overline{\mathcal{T}_{\nu}}$, with similar results for the remaining $i$.\\

Therefore, if $F$ satisfies $(*)$, we must have;\\

$F(0,x)=f(x)$, $(x\in\overline{\mathcal{V}_{\eta}})$\\

$F({i+1\over\nu},-\pi+\pi{j\over\eta})$\\

$=F({i\over\nu},-\pi+\pi{j\over\eta})+{\eta^{2}\over 4\pi^{2}\nu}(F({i\over\nu},-\pi+\pi{j+2\over\eta})-2F({i\over\nu},-\pi+\pi{j\over\eta})+F({i\over\nu},-\pi+\pi{j-2\over\eta}))$\\

$={\eta^{2}\over 4\pi^{2}\nu}F({i\over\nu},-\pi+\pi{j+2\over\eta})+(1-{\eta^{2}\over 2\pi^{2}\nu})(F({i\over\nu},-\pi+\pi{j\over\eta})+{\eta^{2}\over 4\pi^{2}\nu}F({i\over\nu},-\pi+\pi{j-2\over\eta}))$ $(*)$\\

$(1\leq i\leq \nu^{2}-2,0\leq j\leq 2\eta-1)$\\

See also the proof of Lemma 0.5 in \cite{dep3}. The choice of $\eta$ ensures that $1-{\eta^{2}\over 2\pi^{2}\nu}\geq 0$. Hence, inductively, if $|F_{i\over\nu}|\leq M$, then, by $(*)$;\\

$|F_{i+1\over\nu}|\leq M({\eta^{2}\over 4\pi^{2}\nu}+(1-{\eta^{2}\over 2\pi^{2}\nu})+{\eta^{2}\over 4\pi^{2}\nu})=M$.\\

We can differentiate $(*)$ and replace $F$ with ${\partial F\over\partial x}$ or ${\partial^{2} F\over\partial x^{2}}$. The same argument, and the assumption on the initial conditions, gives the required bound.

\end{proof}

\begin{lemma}
\label{initialbounded}
If $f\in C^{\infty}[-\pi,\pi]$, and $f_{\eta}$ is defined on $\overline{\mathcal{V}_{\eta}}$ by;\\

$f_{\eta}(-\pi+\pi{j\over\eta})=f^{*}(-\pi+\pi{j\over\eta})$\\

$f_{\eta}(x)=f(-\pi+{\pi\over\eta}[{\eta(x+\pi)\over \pi}]$\\

where $f^{*}$ is the transfer of $f$ to ${^{*}[-\pi,\pi)}$, then there exists a constant $M\in\mathcal{R}$, such that $max\{f_{\eta},f_{\eta}',f_{\eta}''\}\leq M$. In particular, if $F$ solves the nonstandard heat equation, with initial condition $f_{\eta}$, then, $max\{F,{\partial F\over\partial x},{\partial^{2} F\over\partial x^{2}}\}\leq M$ as well.

\end{lemma}

\begin{proof}
We have, for $x\in [-\pi,\pi)$, using Taylor's Theorem, that;\\

$|{1\over 2h}(f(x+h)-f(x-h))-f'(x)|$\\

$={1\over 2h}(f(x)+hf'(x)+{h^{2}\over 2}f''(c)-f(x)+hf'(x)-{h^{2}\over 2}f''(c'))-f'(x)|$\\

$\leq hK$\\

where $K=max_{[-\pi,\pi)}f''$. By transfer, it follows, that, for infinite $\eta$, $(f_{\eta})'\simeq (f')_{\eta}$. Clearly $(f')_{\eta}$ is bounded, as $f'$ is, which gives the result for $(f_{\eta})'$. The case for $(f_{\eta})''$ is similar. The final result is immediate from Lemma \ref{nsheat}.
\end{proof}

\begin{defn}
\label{nstransf}
We let $\overline{\mathcal{Z}_{\eta}}=\{m\in{{^{*}}\mathcal{Z}}:-\eta\leq m\leq \eta-1\}$ Given a measurable $f:\overline{\mathcal{V}_{\eta}}\rightarrow{^{*}\mathcal{C}}$, we define, for $m\in\mathcal{Z}_{\eta}$, the $m'th$ discrete Fourier coefficient to be;\\

$\hat{f}_{\eta}(m)={1\over 2\pi}\int_{\overline{\mathcal{V}_{\eta}}}f(y)exp_{\eta}(-iym)d\mu_{\eta}(y)$\\

Transposing Lemma 0.9 of \cite{dep5}, (\footnote{\label{measure}\emph{We have there that the measure on} $\overline{\mathcal{S}_{\eta}}=\lambda_{\eta}$. The result follows using the scalar map $p:\overline{\mathcal{V}_{\eta}}\rightarrow\overline{\mathcal{S}_{\eta}}$, $p(x)={x\over\pi}$, and the fact that $p_{*}(\mu_{\eta})=\lambda_{\eta}$});\\

$f(x)=\sum_{m\in\mathcal{Z}_{\eta}}\hat{f}_{\eta}(m)exp_{\eta}(ixm)$ $(*)$\\

Given a measurable $f:\overline{\mathcal{S}_{\eta,\nu}}\rightarrow{^{*}\mathcal{C}}$, we define the nonstandard vertical Fourier transform $\hat{f}:\overline{\mathcal{T}_{\nu}}\times\overline{\mathcal{Z}_{\eta}}\rightarrow{^{*}\mathcal{C}}$ by;\\

$\hat{f}(t,m)={1\over 2\pi}\int_{\overline{\mathcal{V}_{\eta}}}f(t,x)exp_{\eta}(-ixm)d\mu_{\eta}(x)$\\

and, given a measurable $g:\overline{\mathcal{T}_{\nu}}\times\overline{\mathcal{Z}_{\eta}}\rightarrow{^{*}\mathcal{C}}$, we define the nonstandard inverse vertical Fourier transform by;\\

$\check{g}(t,x)=\sum_{m\in\mathcal{Z}_{\eta}}g(t,m)exp_{\eta}(ixm)$\\

so that, by $(*)$, $f=\check{\hat{f}}$\\

Similar to Definition 0.20 of \cite{dep4}, for $f\in\overline{\mathcal{V_{\eta}}}$, we let $\phi_{\eta}:\overline{\mathcal{Z}_{\eta}}\rightarrow{^{*}\mathcal{C}}$ be defined by;\\

$\phi_{\eta}(m)={\eta\over 2\pi}(exp_{\eta}(-im{\pi\over\eta})-exp_{\eta}(im{\pi\over\eta}))$\\

We let $\psi_{\eta}:\overline{\mathcal{Z}_{{\eta\over 2}}}\rightarrow{^{*}\mathcal{C}}$ be defined by;\\

$\psi_{\eta}(m)={\eta\over 2\pi}(1-exp_{\eta}(im{2\pi\over\eta}))$\\

and, we let $U_{\eta}:\overline{\mathcal{Z}_{{\eta\over 2}}}\rightarrow{^{*}\mathcal{C}}$ be defined by;\\

$U_{\eta}(m)=exp_{\eta}(-im{2\pi\over\eta}))$\\

\end{defn}

The following is the analogue of Lemma 0.14 in \cite{dep5}, using the definition of the discrete derivative in Definition \ref{partials} and the discrete Fourier coefficients from Definition \ref{nstransf};\\

\begin{lemma}
\label{sine}
Let $f:\overline{\mathcal{V}_{\eta}}\rightarrow{^{*}\mathcal{C}}$ be measurable; then, for $m\in\mathcal{Z}_{\eta}$,\\

$\hat{f''}(m)=\phi_{\eta}^{2}(m)\hat{f}(m)$\\

\end{lemma}

\begin{proof}

We have, using Lemma \ref{discretederivativeshift}(iii), that;\\

$(\hat{f'})(m)={1\over 2\pi}\int_{\overline{\mathcal{V}_{\eta}}}f'(x)exp_{\eta}(-ixm)d\mu_{\eta}(y)$\\

$=-{1\over 2\pi}\int_{\overline{\mathcal{V}_{\eta}}}f(x)(exp_{\eta})'(-ixm)d\mu_{\eta}(x)$\\

A simple calculation shows that;\\

$(exp_{\eta})'(-ixm)=exp_{\eta}(-ixm)\phi_{\eta}(m)$\\

Therefore;\\

$(\hat{f'})(m)=-\phi_{\eta}(m)\hat{f}(m)$\\

Then $\hat{f''}(m)$\\

$=-\phi_{\eta}(m)\hat{f'}(m)$\\

$=\phi_{\eta}^{2}(m)\hat{f}(m)$\\

as required.
\end{proof}

\begin{lemma}
\label{restrictedsine}
If $f:\overline{\mathcal{V}_{\eta}}\rightarrow{^{*}\mathcal{C}}$ is measurable, then, for $m\in\mathcal{Z}_{{\eta\over 2}}$, we have that;\\

$\hat{\overline{f''}}(m)=\psi_{\eta}(m)^{2}U_{\eta}(m)(\hat{\overline{f}}(m))$\\

\end{lemma}

\begin{proof}
 we have, using (iii) in footnote 3, that;\\

$\hat{\overline{f'}}(m)={1\over 2\pi}\int_{\overline{\mathcal{V}_{{\eta\over 2}}}}\overline{f'}(x)exp_{{\eta\over 2}}(-ixm)d\mu_{{\eta\over 2}}(x)$\\

$={1\over 2\pi}\int_{\overline{\mathcal{V}_{{\eta\over 2}}}}\overline{f'}(x)\overline{exp_{\eta}}(-ixm)d\mu_{{\eta\over 2}}(x)$\\

$={1\over 2\pi}\int_{\overline{\mathcal{V}_{{\eta\over 2}}}}\overline{f^{rsh}}(x)\overline{{exp'_{\eta}}^{rsh}}(-ixm)d\mu_{{\eta\over 2}}(x)$\\

We calculate;\\

$\overline{{exp'_{\eta}}^{rsh}}(-im(-\pi+\pi{2j\over\eta}))$\\

$={exp'_{\eta}}(-im(-\pi+\pi{2j-1\over\eta}))$\\

$={\eta\over 2\pi}({exp_{\eta}}(-im(-\pi+\pi{2j\over\eta}))-{exp_{\eta}}(-im(-\pi+\pi{2j-2\over\eta})))$\\

$={exp_{\eta}}(-im(-\pi+\pi{2j\over\eta}))[{\eta\over 2\pi}(1-exp_{\eta}(im({2\pi\over\eta})))]$\\

$=\psi_{\eta}(m)\overline{exp_{\eta}}(-im(-\pi+\pi{2j\over\eta}))$\\

$\hat{\overline{f'}}(m)={1\over 2\pi}\psi_{\eta}(m)\int_{\overline{\mathcal{V}_{{\eta\over 2}}}}\overline{f^{rsh}}(x)\overline{{exp_{\eta}}}(-ixm)d\mu_{{\eta\over 2}}(x)$\\

$=\psi_{\eta}(m)(\hat{\overline{f^{rsh}}}(m))$\\

It follows that;\\

$\hat{\overline{f''}}(m)=\psi_{\eta}(m)(\hat{\overline{(f')^{rsh}}}(m))$\\

$=\psi_{\eta}(m)(\hat{\overline{(f^{rsh})'}}(m))$\\

$=\psi_{\eta}(m)^{2}(\hat{\overline{f^{rsh^{2}}}}(m))$\\

We calculate;\\

$\hat{\overline{f^{rsh^{2}}}}(m)$\\

$={1\over 2\pi}\int_{\overline{\mathcal{V}_{{\eta\over 2}}}}\overline{f^{rsh^{2}}}(x)\overline{exp_{\eta}}(-ixm)d\mu_{{\eta\over 2}}(x)$\\

$={1\over 2\pi}exp_{\eta}({-2\pi im\over\eta})\int_{\overline{\mathcal{V}_{{\eta\over 2}}}}\overline{f^{rsh^{2}}}(x)\overline{exp_{\eta}^{rsh^{2}}}(-ixm)d\mu_{{\eta\over 2}}(x)$\\

$={1\over 2\pi}U_{\eta}(m)\int_{\overline{\mathcal{V}_{{\eta\over 2}}}}\overline{f}(x)\overline{exp_{\eta}}(-ixm)d\mu_{{\eta\over 2}}(x)$\\

$=U_{\eta}(m)\hat{\overline{f}}(m)$\\

Hence;\\

$\hat{\overline{f''}}(m)=\psi_{\eta}(m)^{2}U_{\eta}(m)(\hat{\overline{f}}(m))$\\

as required.

\end{proof}

\begin{lemma}
\label{decayrate}

If $f\in V(\overline{\mathcal{V}_{\eta}})$, with $f''$ bounded, then, there exists a constant $F\in\mathcal{R}$, with;\\

$|\hat{\overline{f}}(m)|\leq {F\over m^{2}}$, for $m\in\mathcal{Z}_{{\eta\over 2}}$.\\

Moreover;\\

$({^{\circ}\overline{f}})(x)=\sum_{m\in\mathcal{Z}}{^{\circ}((\hat{\overline{f}})(m))}exp(im{^{\circ}x})$, $x\in\overline{\mathcal{V}_{\eta}}$.\\

\end{lemma}

\begin{proof}
Using results of \cite{dep5}, we have that $m\leq |\psi_{\eta}(m)|\leq 2m$, and $|U_{\eta}(m)|=1$ for $|m|\leq{\eta\over 2}$. As $f''$ is bounded, so is $\overline{f''}$, so  $|\hat{\overline{f''}}(m)|\leq F\in\mathcal{R}$. This implies, by the result of Lemma \ref{restrictedsine} that;\\

$|\hat{\overline{f}}(m)|\leq {F\over m^{2}}$. $(*)$\\

for $m\in\mathcal{Z}_{\eta\over 2}$, as required. Using the Inversion Theorem from Definition \ref{nstransf}, we have that;\\

$\overline{f}(x)={^{*}\sum}_{m\in\mathcal{Z}_{\eta\over 2}}\hat{\overline{f}}(m)exp_{{\eta\over 2}}(ixm)=M$, $(x\in\overline{\mathcal{V}_{\eta}})$\\

Let $L=\sum_{m\in\mathcal{Z}}{^{\circ}(\hat{\overline{f}}(m))}exp(im{^{\circ}x})$\\

If $\epsilon>0$, we have, using the result of \ref{decayrate}, and the fact that $exp_{{\eta\over 2}}$ is $S$-continuous for $m\in\mathcal{Z}$,  that;\\

$|M-L|\leq |M-M_{n}|+|M_{n}-L_{n}|+|L-L_{n}|$\\

$\leq {^{*}\sum}_{n+1\leq |m|\leq {\eta\over 2}}{F\over m^{2}}+\sum_{m=1}^{n}\delta_{i}+\sum_{|m|\geq n+1}{F+1\over m^{2}}$ $(\delta_{i}\simeq 0)$\\

$\leq {2F({1\over n}-{2\over \eta})}+\delta+{2(F+1)\over n}$ $(\delta\simeq 0)$\\

$\leq {4(F+1)\over n}<\epsilon$\\

for $n>{\epsilon\over 4(F+1)}$, $n\in\mathcal{N}$. As $\epsilon$ was arbitrary, we obtain the result.

\end{proof}

\begin{lemma}
\label{coefficient}
If $F$ solves the nonstandard heat equation, with initial condition $f$, bounded and $S$-continuous, such that ${^{\circ}f}(x)=g({^{\circ}x})$, where $g$ is continuous and bounded on $[-\pi,\pi]$, then;\\

${^{\circ}(\hat{\overline{F}}(m,t))}=e^{-m^{2}{^{\circ}t}}(\hat{g})(m)$\\

for $m\in\mathcal{Z}$ and finite $t$.\\

\end{lemma}

\begin{proof}
As ${\partial F\over \partial t}-{\partial^{2}F\over \partial x^{2}}=0$, we have, taking restrictions, that;\\

$\overline{{\partial F\over \partial t}}-\overline{{\partial^{2}F\over \partial x^{2}}}=0$\\

Taking Fourier coefficients, for $m\in\mathcal{Z}$, and, using Lemma \ref{restrictedsine};\\

${d\hat{\overline{F}}(m,t)\over dt}-\theta_{\eta}(m)\hat{\overline{F}}(m,t)=0$\\

where $\theta_{\eta}(m)=\psi_{\eta}^{2}(m)U_{\eta}(m)$. Then;\\

$\nu(\hat{\overline{F}}(m,t+{1\over \nu})-\hat{\overline{F}(m,t)})=\theta_{\eta}(m)\hat{\overline{F}}(m,t)$\\

Rearranging, we obtain;\\

$\hat{\overline{F}}(m,t+{1\over\nu})=(1+{\theta_{\eta}(m)\over \nu})\hat{\overline{F}}(m,t)$\\

and, solving the recurrence;\\

$\hat{\overline{F}}(m,t)=(1+{\theta_{\eta}(m)\over \nu})^{[\nu t]}\hat{\overline{F}}(m,0)$\\

Taking standard parts, and using the facts that $lim_{n\rightarrow\infty}(1+{x\over n})^{n}=e^{x}$, and ${^{\circ}{\theta_{\eta}(m)}}=-m^{2}$, we obtain;\\

${^{\circ}(\hat{\overline{F}}(m,t))}=(e^{-m^{2}{^{\circ}t}}){^{\circ}}(\hat{\overline{f}}(m))$\\

for finite $t$. As $f$ is bounded and $S$-continuous, so is $\overline{f}$, and ${^{\circ}f}={^{\circ}\overline{f}}$ is integrable. We have that;\\

${^{\circ}\int_{\overline{\mathcal{V}_{\eta\over 2}}}\overline{f}(x)exp_{\eta\over 2}(-imx)}d\mu_{\eta\over 2}=\int_{-\pi}^{\pi}{^{\circ}f({^{\circ} x})}e^{-ixm}dx=\int_{-\pi}^{\pi}g(x)e^{-ixm}dx$\\

Hence;\\

${^{\circ}(\hat{\overline{F}}(m,t))}=(e^{-m^{2}{^{\circ}t}}\hat{g}(m))$\\

as required.\\

\end{proof}

\begin{theorem}
Let $g\in C^{\infty}([-\pi,\pi])$, and $G$ be as in Lemma \ref{heatequation}. Let $f=g_{\eta}$, and let $F$ be as in Lemma \ref{nsheat}. Then, for finite $t$, and $(x,t)\in\overline{\mathcal{V}_{\eta}}\times\overline{ \mathcal{T}_{\nu}}$, ${^{\circ}F}(x,t)=G({^{\circ}x},{^{\circ}t})$

\end{theorem}

\begin{proof}
By Lemma \ref{initialbounded}, we have that ${\partial^{2}F\over\partial x^{2}}$ is bounded. By Lemma \ref{decayrate};\\

${^{\circ}\overline{F}(x,t)}=\sum_{m\in\mathcal{Z}}{^{\circ}(\hat{\overline{F}}(m,t))}exp(im{^{\circ}x})$ $(*)$\\

By Lemma \ref{coefficient};\\

${^{\circ}(\hat{\overline{F}}(m,t))}=e^{-m^{2}{^{\circ}t}}\mathcal{F}(g)(m)$ $(**)$\\

Comparing $(*),(**)$, with the expression;\\

$G({^{\circ} x},{^{\circ}t})=\sum_{m\in\mathcal{Z}}e^{-m^{2}{^{\circ}t}}\mathcal{F}(g)(m)e^{im{^{\circ}x}}$\\

obtained in Lemma \ref{heatequation}, gives the result that ${^{\circ}\overline{F}(x,t)}=G({^{\circ}x},{^{\circ}t})$. However, $\overline{F}_{t}$ is $S$-continuous, for finite $t$, by the fact that $({\partial F\over \partial x})_{t}$ is bounded, from Lemma \ref{initialbounded}, hence;\\

${^{\circ}\overline{F}(x,t)}={^{\circ}F(x,t)}=G({^{\circ}x},{^{\circ}t})$\\

as required.\\

\end{proof}

\begin{theorem}
\label{equilibrium}
Let $g\in C^{\infty}([-\pi,\pi])$, and $G$ be as in Lemma \ref{heatequation}. Let $f=g_{\eta}$, and let $F$ be as in Lemma \ref{nsheat}. Then, for infinite $t$, and $(x,t)\in\overline{\mathcal{V}_{\eta}}\times\overline{ \mathcal{T}_{\nu}}$, $F(x,t)\simeq \int_{\overline{\mathcal{V}_{\eta}}} f d\mu_{\eta}$.

\end{theorem}

\begin{proof}
Again, using Lemma \ref{initialbounded} and lemma \ref{decayrate}, we have, using the proof of \ref{coefficient}, that;\\

$\overline{F}(x,t)\simeq \sum_{m\in\mathcal{Z}}exp_{\nu}^{-\theta_{\eta}(m)t}\hat{f}(m)exp_{\eta\over 2}^{imx}$\\

Taking standard parts, using lemma \ref{decayrate}, and the fact that $exp_{\nu}^{-\theta_{\eta}(m)t}\simeq 0$, for finite $m$ and infinite $t$, we see that all the coefficients vanish, except when $m=0$, that is;

$\overline{F}(x,t)\simeq \hat{f}(0)=\int_{\overline{\mathcal{V}_{{\eta\over 2}}}}\overline{f}d\mu_{{\eta\over 2}}\simeq \int_{\overline{\mathcal{V}_{\eta}}}f d\mu_{\eta} $\\

The result follows for $F(x,t)$ by $S$-continuity.\\

\end{proof}

\begin{rmk}
When $\eta={2\pi\sqrt{\nu}\over\sqrt{3}}$, we obtain, by Lemma \ref{nsheat}, the iterative scheme for the nonstandard Markov chain with transition probabilities $\{{1\over 3},{1\over 3},{1\over 3}\}$. By Lemma \ref{equilibrium}, we obtain convergence to equilibrium after at least $\nu^{2}={9\eta^{4}\over 16\pi^{2}}$ steps, which is polynomial in the number of states $\eta$. This is a considerable improvement over the result  in Lemma \ref{nonstandard1}, which is exponential in $\eta$. The discrepancy results from the choice of a "smooth" initial distribution. The method of reverse martingales is useful to consider other "nonsmooth" cases, for which a Fourier analysis is impossible.

\end{rmk}

\end{document}